\documentclass[11pt]{amsart}

\usepackage{amssymb,graphicx, amsmath, amsthm, enumitem,MnSymbol, array,mathalpha,enumitem}
\usepackage{calrsfs}
\usepackage{wrapfig}
\usepackage{graphicx}
\graphicspath{ {./images/} }

\DeclareMathAlphabet\mathbfcal{OMS}{cmsy}{b}{n}

\usepackage{todonotes}

\usepackage{floatflt}

\usepackage{floatrow}

\usepackage{hyperref}

\usepackage{mathbbol}

 \usepackage[mathscr]{eucal}

\usepackage[font=small]{caption}

\def\PP{\mathsf{P}}
\def\AA{\mathsf{A}}
\def\BB{\mathsf{B}}
\def\CC{\mathsf{C}}

\def\Q{{\mathcal{Q}}} 
\def\k{\mathbb R} 

\def\k{{\kappa}}

\def\P{{\mathbb{P}}}
\def\mid{{\rm mid}}

\def\B{{\mathbb{B}}}

\def\A{\mathbb{A}}

\begin{document}

\newtheorem{theorem}{Theorem}[section]
\newtheorem{lemma}[theorem]{Lemma}
\newtheorem{proposition}[theorem]{Proposition}
\newtheorem{corollary}[theorem]{Corollary}
\newtheorem{problem}[theorem]{Problem}
\newtheorem{construction}[theorem]{Construction}

\theoremstyle{definition}
\newtheorem{defi}[theorem]{Definitions}
\newtheorem{definition}[theorem]{Definition}
\newtheorem{notation}[theorem]{Notation}
\newtheorem{remark}[theorem]{Remark}
\newtheorem{example}[theorem]{Example}
\newtheorem{question}[theorem]{Question}
\newtheorem{comment}[theorem]{Comment}
\newtheorem{comments}[theorem]{Comments}

\newtheorem{discussion}[theorem]{Discussion}

\renewcommand{\thedefi}{}

\long\def\alert#1{\smallskip{\hskip\parindent\vrule%
\vbox{\advance\hsize-2\parindent\hrule\smallskip\parindent.4\parindent%
\narrower\noindent#1\smallskip\hrule}\vrule\hfill}\smallskip}

\def\ff{\frak}
\def\tf{torsion-free}
\def\Spec{\mbox{\rm Spec }}
\def\Proj{\mbox{\rm Proj }}
\def\hgt{\mbox{\rm ht }}
\def\type{\mbox{ type}}
\def\Hom{\mbox{ Hom}}
\def\rank{\mbox{rank}}
\def\Ext{\mbox{ Ext}}
\def\Tor{\mbox{ Tor}}
\def\ker{\mbox{ Ker }}
\def\Max{\mbox{\rm Max}}
\def\End{\mbox{\rm End}}
\def\xpd{\mbox{\rm xpd}}
\def\Ass{\mbox{\rm Ass}}
\def\emdim{\mbox{\rm emdim}}
\def\epd{\mbox{\rm epd}}
\def\repd{\mbox{\rm rpd}}
\def\ord{\mbox{\rm ord}}

\def\htt{\mbox{\rm ht}}

\def\DD{{\mathcal D}}
\def\EE{{\mathcal E}}
\def\FF{{\mathcal F}}
\def\GG{{\mathcal G}}
\def\HH{{\mathcal H}}
\def\II{{\mathcal I}}
\def\LL{{\mathcal L}}
\def\MM{{\mathcal M}}

\def\k{\mathbb{k}}





\title[Bisector fields and    pencils of conics]{Bisector fields and    pencils of conics}

\author{Bruce Olberding$^*$} 
\address{Department of Mathematical Sciences, New Mexico State University, Las Cruces, NM 88003-8001}

\email{bruce@nmsu.edu}

\author{Elaine A.~Walker}
\address{Las Cruces, NM 88011}

\email{miselaineeous@yahoo.com}

\begin{abstract}   
We introduce the notion of a bisector field, which is a maximal collection of pairs of lines such that for each line in each pair,  the midpoint of the points where the line crosses every pair is the same, regardless of choice of pair.  We use this 
to study asymptotic properties of pencils of affine conics over fields and  show 
 that  pairs of lines in the plane that occur as the   asymptotes of hyperbolas from a pencil of affine conics belong to a bisector field. By including also pairs of parallel lines arising from degenerate parabolas in the pencil, we obtain a full characterization: 
Every bisector field arises  from a pencil of affine conics, and vice versa, every nontrivial pencil of affine conics is asymptotically a bisector field. Our main results are valid  over any field of characteristic other than $2$ and hence hold in the classical Euclidean setting as well as in Galois geometries. 
 
%
%
%

 \end{abstract}



\thanks{$^*$Corresponding author}

\thanks{\today}


\maketitle

\section{Motivation}

This article focuses on the question of which pairs of lines occur as the asymptotes of hyperbolas in a pencil of affine conics.   
 We were led to this question from several directions because the solution we give to it, the bisector fields of the title, appeared for us as a solution to other seemingly unrelated questions, and the connection between this idea and asymptotes of hyperbolas in pencils seems not to appear in the literature.
%
In this paper, we introduce bisector fields and establish their connection to pencils of conics. In sequels to this article we consider other applications of bisector fields, some of which we describe later in this section in order to help motivate the concept.

   Beltrami \cite[Section~2]{Beltrami}  (but see also  \cite{Vac})    dealt with what can be considered a projective version of our asymptotes question.  We mention it here  by way of contrast with our results.  Given a pencil of projective conics and a line $\ell$ in the projective plane, Beltrami shows the envelope of the pairs of lines that are tangent to 
a conic in the pencil at the points where $\ell$ crosses the conic  is  
a  rational quartic curve with three cusps, possibly two of which are imaginary. Thus, with $\ell$ designated the line at infinity and everything else restricted to the resulting affine chart, 
the asymptotes of the hyperbolas in the  pencil  are the lines tangent to a rational quartic curve with three cusps, at least one of which is real. The cusps of this envelope are visible in Figures 1(a) and 1(b).

Our approach
ultimately does give a different proof of Beltrami's theorem, but we postpone the connection with his tricuspidal quartic curve to the article \cite{OW6}, where, using the results of the present paper and projective duality, we arrive at his theorem from a rather different and more general  starting place involving planar arrangements of lines. The aim in the present article is  to give a very different type of criterion for when pairs of lines occur as asymptotes of hyperbolas in a pencil of affine conics, a criterion that involves only how the  lines cross each other.    In our approach, 
 the underlying field~$\k$ seldom matters, other than that it has characteristic not equal to $2$. (We avoid characteristic $2$ because    we work with midpoints.)
  Our main theorem is that whether a set $\A$ of pairs of intersecting lines in the affine plane over $\k$ is the set of asymptotes for the hyperbolas in a  pencil is detectable from an unusual feature of the geometry of the pairs, namely that each line  in each pair simultaneously bisects  all other pairs in the set~$\A$. 
   What bisects means is that the midpoint of the points where the line crosses a pair of lines is the same regardless of the pair that is chosen from~$\A$.  
   
   To state the main theorem more precisely, we need to consider not only  asymptotes of hyperbolas from a pencil but  pairs of parallel lines that share a midline with a pair of parallel lines in the pencil. Taken together, the asymptotes and the pairs of parallel lines are the degenerate conics in an affine net generated by the pencil and the line at infinity. We call such a  collection of pairs of lines an {\it asymptotic pencil} since it represents the asymptotic behavior of a pencil of affine conics. The main theorem 
  of the article is then that a set  of pairs of lines is a nontrivial asymptotic pencil    if and only if  
       it is   a {\it bisector field}, 
      a collection $\B$ of pairs of lines such that each line  in each pair simultaneously bisects  all other pairs in the set $\B$ and such that $\B$ is maximal with respect to this bisection property---maximal in the sense that $\B$ cannot be enlarged without destroying the bisection property. 
      Figures~1 and~2 give examples of pencils and their resulting bisector fields. 
      
         \begin{figure}[h] 
     \label{9noobconics}
 \begin{center}
\includegraphics[width=1\textwidth,scale=.09]{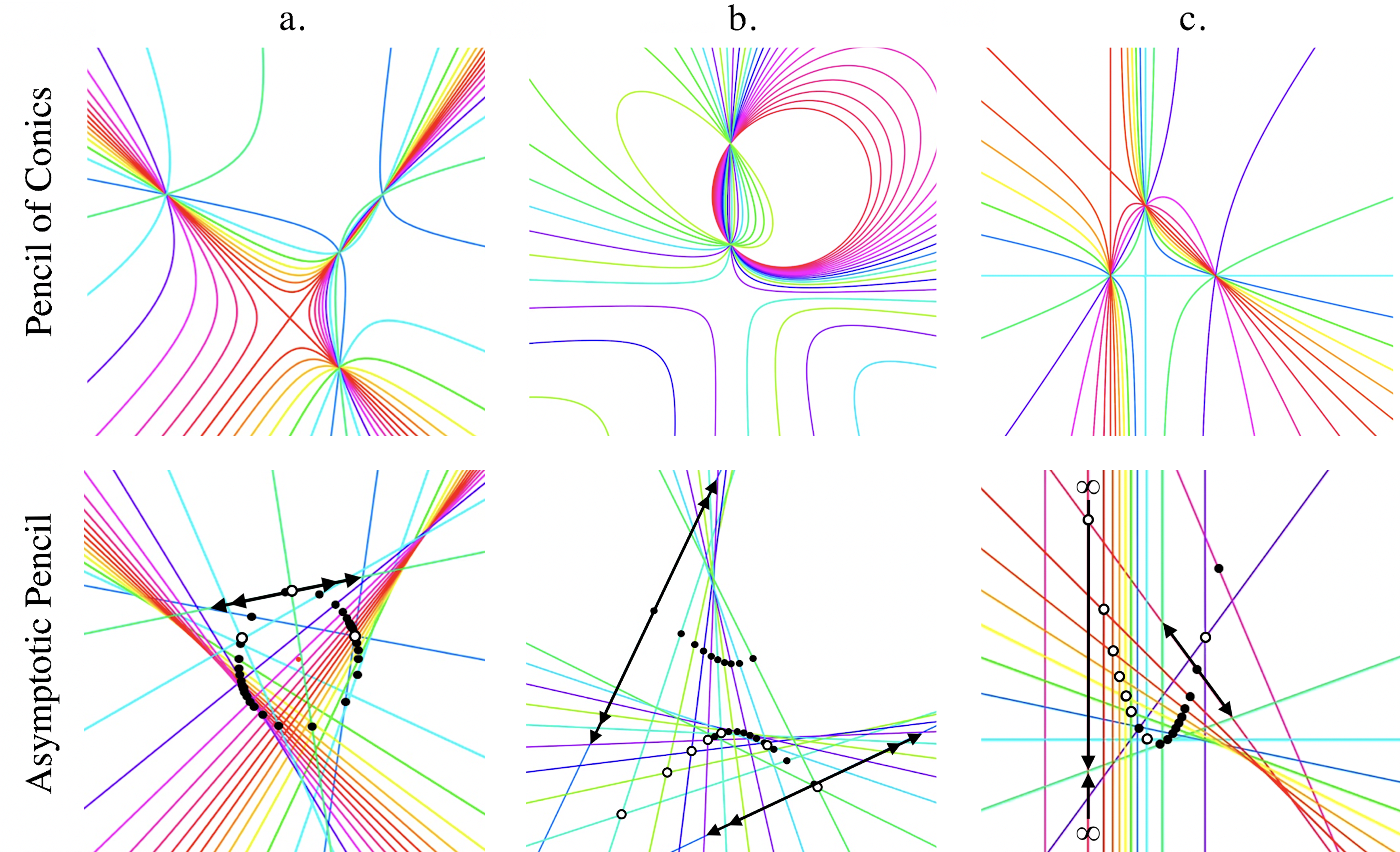} 
 \end{center}
 \caption{Selected conics from three pencils of conics (top row) and their asymptotic pencils (bottom row). The asymptotes of each hyperbola in the top row appear in the asymptotic pencil below.  By Theorem~\ref{main theorem 1}, the asymptotic pencils are bisector fields. The black points in the bisector fields are the midpoints of the bisectors; each white point is an intersection of the lines in a  bisector pair, which is the center of the degenerate hyperbola. The black arrows in the bottom row help illustrate the bisection property. }
\end{figure}

    \begin{figure}[h] 
     \label{9noobconics}
 \begin{center}
\includegraphics[width=1\textwidth,scale=.09]{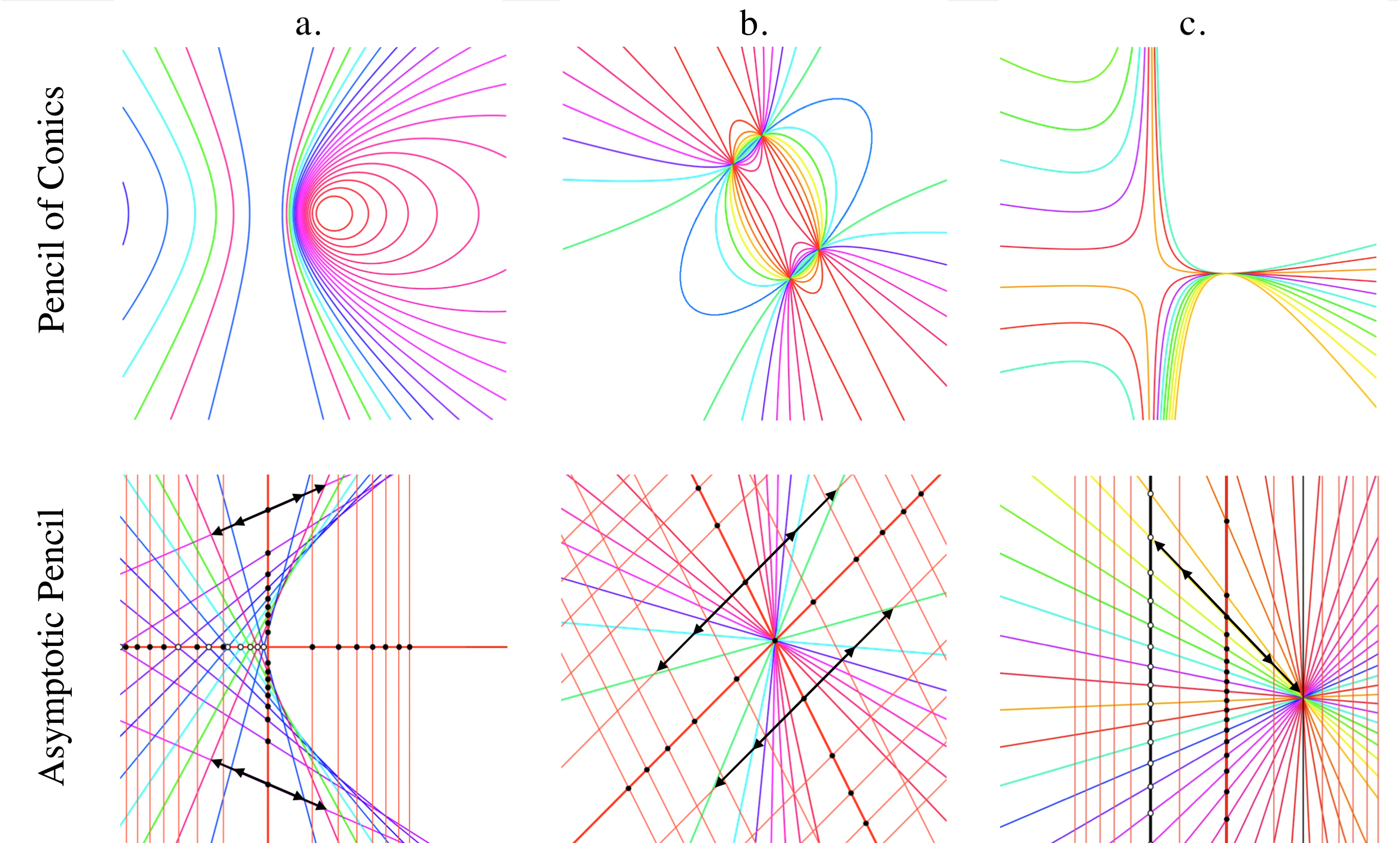} 
 \end{center}
 \caption{
 As in Figure 1, the top row contains pencils of conics and the bottom row the corresponding asymptotic pencils and bisector fields. 
These examples differ from Figure 1 in that each pencil contains reducible conics that are pairs of parallel lines. Each asymptotic pencil that contains a pair of parallel lines contains all pairs of parallel lines having the same midline, a property that is visible in the examples in the second row. The midline itself is a double line and contains the midpoints of all bisectors not parallel to it.  }
\end{figure}


Bisector fields prove to be useful beyond the present setting, and by way of further motivation
we mention several other applications of these ideas that are worked out in sequels to this paper.   


\smallskip

(1) Continuing to work over a field of characteristic $\ne 2$, in \cite{OW5} we study the bisectors of a quadrilateral, where a line is said to  bisect a quadrilateral if it bisects both pairs of opposite sides. The collection of all bisectors of a quadrilateral admits a pairing given by an orthogonality relation from  an inner product induced by the quadrilateral itself. With this pairing, rather surprisingly, the bisectors form a bisector field---surprising because of the non-obvious symmetry here: although the lines were chosen to bisect the quadrilateral, every pair of bisectors paired with respect to the orthogonality relation not only bisects the quadrilateral but is in turn bisected by every bisector of the quadrilateral. Moreover, the midpoints of the bisectors lie on a conic, which turns out to be the nine-point conic of the quadrilateral, a conic that passes through nine distinguished points of the complete quadrilateral, and which, as 
 discussed in  \cite{Vac}, 
 was studied by Steiner and Beltrami in the mid-nineteenth century and then rediscovered by British geometers several decades later. Thus bisector fields help complete a picture of the nine-point conic by  giving an interpretation of the other points on the nine-point conic.  
 
\smallskip
 

(2) In \cite{OW6}, we describe the dual curve of bisectors and use this to  classify  bisector fields in terms of envelopes of tangent lines. Over the field of real numbers or complex numbers, we obtain a complete classification of all bisector fields by finding the envelope of a bisector field. Beltrami's rational tricuspidal  quartic thus reappears as one possible envelope. Over a finite field, the classification is more complicated and incomplete. In any case, bisector fields  give a new interpretation of the tangent lines of the Steiner deltoid, a curve and its tangent lines that appear in numerous contexts. 
 
\smallskip
 

(3) 
Bisector fields can be viewed as equivalence classes of quadrilaterals  or quadrangles. Call two quadrilaterals   entangled if they have the same bisector field. Entanglement is shown in \cite{OWSpider} to reduce to a simple, {\it a priori} asymmetric, criterion, namely that the sides of one quadrilateral bisect the sides of the other with respect to the same midpoints. 
A quadrilateral is entangled with another quadrilateral if and only if the pencils of conics their vertices define are entangled also, in the sense that each hyperbola in one pencil shares asymptotes with a hyperbola in the other pencil. In \cite{OWSpider}, we study these equivalence classes of quadrilaterals.  
 

\smallskip

(4) If $\B$ is a collection of pairs of lines in the plane and $p$ is a point, then the {\it diametric set} of $\B$ at $p$ is the set of points $q$ for which $q$ and the reflection of $q$ through $p$ lie on a pair in $\B$.  
  It is shown in  \cite{OWSpider}  that
 if the underlying field $\k$ has at least 23 elements, then  $\B$ is a bisector arrangement if and only if the diametric set at each point is a conic. Thus a collection of pairs of lines is a bisector arrangement if and only if  the pairs are ``a conic distance away'' from each point of the plane. We show also in  \cite{OWSpider} that when  $\k =  
  {\mathbb{R}}$, there is an interesting and rather tight connection between  the  diametric conics  of a bisector field  and the geometry of fields of ellipses used to describe the state of polarization of  
an electromagnetic wave in the theory of optics.

\smallskip

















     The outline of the paper is as follows.  
      In Section 2 we establish some terminology for pencils in our affine setting, including the 
     notion of a degeneration of a conic, a slight generalization of the affine concept of asymptotes of a hyperbola. Section~3 develops the notion of an asymptotic pencil, the set of degenerations of the conics in the pencil. It is shown in Corollary~\ref{generators2 cor}
       that regardless of the field, every asymptotic pencil contains a degenerate hyperbola (and at least two if the field has more than 3 elements). When $\k$ is an algebraically closed field, the asymptotic pencil can be viewed as a planar cubic in five-dimensional projective space (Proposition~\ref{cubic}). 
      
    A typical way a pencil of affine conics arises is as the set of conics through four points in general position. These four points are the vertices of a quadrilateral, and in Section~4 quadrilaterals are used to help describe technical  properties of the asymptotic pencil, such as whether two different degenerate hyperbolas in the pencil can share a line (see Lemma~\ref{trap}). 
    These properties are needed in proofs in the next sections.
    
    Bisectors are introduced in Section~5, first as bisectors of collections of conics, then as bisectors of pairs of lines (i.e., degenerate conics) and bisectors of quadrilaterals.  The main theorem of the section, Theorem~\ref{amazing}, implies that a bisector of a pair of conics bisects all the pairs of asymptotes of the hyperbolas in the pencil generated by the conics, as well as all hyperbolas that share their asymptotes with hyperbolas in the pencil. Aspects of this theorem are closely connected with Desargues' Involution Theorem, and in Corollary~\ref{des} we derive this theorem as a corollary. To conclude the article, 
Section 6  contains the proof of the main theorem, Theorem~\ref{main theorem 1}, which shows that bisector fields and asymptotic pencils are the same thing in different guises. 

It follows from  Proposition~\ref{quad exists} and Theorem~\ref{main theorem 1} that the lines in a nontrivial asymptotic pencil are the bisectors of a quadrilateral, and this idea is a starting point for the sequel~\cite{OW5}. 
This notion of a bisector of a quadrilateral is in distinction to the area and perimeter bisectors of 
Berele and Catoiu \cite{BC1,BC2}, but interestingly there appears in their context also an analytic analogue  of the tricuspidal envelope of bisectors that was discussed above and is studied in \cite{OW6}.

We used Maple to generate the figures in the article.

\section{Preliminaries}

 Since we work with arbitrary fields and outside the typical projective setting for pencils and conics, we first establish terminology. 
Throughout the paper $\k$ denotes a field of characteristic other than $2$.  
An {\it affine conic} is the zero locus in $\k^2$ of a quadratic   polynomial in $\k[X,Y]$ (quadratic, for short). 
Throughout the paper, since we work with affine conics only, we will usually drop the adjective ``affine.''
   A quadratic polynomial $f$ need not be determined by its zero set, but even 
 if $\k$ is finite, if there is more than one  zero of $f$, then the zero set uniquely determines~$f$ up to scalar multiple among quadratics \cite[16.1.4, p.~171]{Berger}. 
An {affine conic} is {\it reducible} if  it is the union of two lines over $\k$; equivalently, the quadratic that defines the conic is a product of linear polynomials. 

  A {\it hyperbola} over the field $\k$ is a   conic that has two distinct points at infinity.  
 A hyperbola has a {\it center},  
 a point of symmetry for the zero set of the hyperbola.  
  The {\it asymptotes} of a hyperbola are the two lines through the center of the hyperbola that meet the hyperbola at infinity. 
A {\it parabola} is a conic with one point at infinity and an {\it ellipse} a conic with no points at infinity. A pair of parallel lines or a double line is thus a degenerate parabola.  

For the sake of expediency, in this article we say two quadratics $f_1$ and $f_2$  are {\it independent} if  
  $ \deg(\alpha f_1+\beta f_2)=2 $
  for all 
   $\alpha,\beta \in \k$ where $\alpha$ and $\beta$ are not both $0$.     
       Thus two quadratics are independent if and only if their degree $2$ homogeneous parts are not scalar multiples of each other; if and only if over the algebraic closure of $\k$, the zero sets of the two quadratics do not share the same points at infinity. 
A {\it pencil of affine conics} is a set of conics that are the zero sets of the nonzero polynomials in $\k f_1 + \k f_2 = \{\alpha f_1 +  \beta f_2:\alpha, \beta \in \k\}
$ for some independent quadratics $f_1,f_2$. (Since $f_1$ and $f_2$ are independent, all the nonzero polynomials in $\k f_1 + \k f_2$ are quadratics.) Thus a pencil of affine conics is parameterized by the projective line over~$\k$.

If two conics in a pencil intersect in a point then all conics in the pencil intersect in this point.
%
Moreover, if as in Figures~1(a) and~2(b) two conics in the pencil intersect in four distinct points, no three of which are collinear, then the pencil is the set of conics through these four points, the {\it basepoints} of the pencil.  

 We need a slight generalization of   asymptotes:

 \begin{definition} 
 A {\it degeneration of a quadratic} $f$ is a  quadratic $g$ that is reducible over $\k$ and for which $f -g \in \k$. A {\it degeneration of an affine conic} is the zero set of a degeneration of the defining quadratic of the conic. 
 \end{definition} 
 
 Thus a degeneration of a conic is a pair of lines, possibly parallel and possibly a double line. 
 Every conic has a degeneration over the algebraic closure of $\k$, but this need not be the case over $\k$ itself, as is evident from the next proposition.


\begin{proposition}  \label{late lemma} 
 A hyperbola has a unique degeneration, namely its asymptotes. A degenerate parabola, i.e., a
 pair of parallel lines $\ell,\ell'$,  has as its degenerations the pairs of lines parallel to $\ell$ and $\ell'$ that share the same midline as the pair  $\ell,\ell'$.   Neither an ellipse nor a nondegenerate parabola  has a degeneration.
 \end{proposition}

 
 \begin{proof} 
 Suppose a conic has a degeneration into a pair of lines defined by linear polynomials $g_1$ and $g_2$, and suppose 
 $g_1g_2 -h_1h_2  \in \k$, where the $h_i$ are linear polynomials in $\k[X,Y]$. If $g_1$ and $g_2$ define lines that are not parallel, then after a change of variables we can assume  $g_1(X,Y) =X$ and $g_2(X,Y) = Y$. 
 Since $XY+\lambda$ is irreducible for all $\lambda \ne 0$, 
 the only degeneration of $g_1g_2$ is itself. 
 It follows that the
  pair of asymptotes of a hyperbola is the unique degeneration of the hyperbola. 
%
 %

If instead $g_1$ and $g_2$ define parallel lines, then   we can assume   $g_1(X,Y) = X$ and $g_2(X,Y)=X-a.$ 
The degenerations of $g_1g_2$ are the conics of the form $(X-r)(X-s)$, where $r,s \in \k$ and $a = r+s$. The midline of the two lines $X=r$ and $X=s$ is $X=a/2$, which is the midline of the lines $g_1$ and $g_2$. 
%
This proves the second assertion of the proposition. 

That an ellipse does not have a degeneration is clear since a conic and its degeneration share the same points at infinity and an ellipse does not meet the line at infinity. 
To see that a nondegenerate parabola $f$ does not have a degeneration, use an affine transformation to reduce to the case that $f(X,Y) = Y-X^2$. Then $f+\lambda$ is an irreducible polynomial for all $\lambda \in \k$, and so $f$ does not have a degeneration.   \end{proof}

    \section{Asymptotic pencils}
    
    As noted in the last section, we will typically drop the ``affine'' in ``affine conic'' since we focus in this article only on the affine case. 
    Our main object of interest is that of an asymptotic pencil.
    
    \begin{definition} 
     A set $\A$ of reducible conics (i.e., pairs of lines) is an {\it asymptotic pencil} 
if there is a pencil of conics for which $\A$ consists of the pairs of asymptotes of the hyperbolas in the pencil and the pairs of parallel lines that share a midline with a pair of parallel lines in the pencil, if any such pairs exist.  
\end{definition}

 The next proposition, which is an immediate consequence of Proposition~\ref{late lemma}, gives a simple   interpretation of an asymptotic pencil that will be useful in proofs.

\begin{proposition} A collection  of pairs of lines is an 
 asymptotic pencil if and only if
it is the set of degenerations of conics in a pencil of conics. 
 \end{proposition}
 

   Figures 1 and 2 give examples of asymptotic pencils.  By definition, independent quadratics 
generate pencils of conics, and the next lemma shows that any choice of two independent quadratics in the pencil will do.

  \begin{lemma} \label{generators} Let $f_1,f_2$ be independent quadratics. 
  Any two pairs of independent quadratics in
  $ \k f_1 + \k f_2$ generate the same pencil, and any two 
pairs of independent quadratics in
    $
  \k f_1 + \k f_2 + \k$ generate the same asymptotic pencil.
  
\end{lemma}

\begin{proof} The argument is routine.  Let $g_1,g_2$ be  independent quadratics in 
$
  \k f_1 + \k f_2 + \k$.
There are $a_{ij} \in \k$ such that 
$g_i = a_{i1}f_1+a_{i2}f_2 + a_{i3}$ for $i = 1,2$. 
If $a_{11}a_{22} -a_{12}a_{21} =0$, then there are $\alpha,\beta \in \k$, not both $0$, such that $\alpha a_{11} = \beta a_{21}$ and $\alpha a_{21} = \beta a_{22}$, and hence $\alpha g_1 -\alpha a_{31}   = \beta g_2- \beta a_{31} $, contradicting the independence of the pair $g_1,g_2$. Thus $a_{11}a_{22} -a_{12}a_{21} \ne 0$, from which it follows that there are $b_{ij} \in \k$ such that 
$f_i = b_{i1}g_1+b_{i2}g_2 + b_{i3}$ for $i = 1,2$.
%
This implies $ \k f_1 + \k f_2 + \k  = 
 \k g_1 + \k g_2 + \k$, and hence the asymptotic pencil generated  by $f_1,f_2$ is the same at that generated by $g_1,g_2$.   Choosing instead $g_1$ and $g_2$  in $\k f_1 + \k f_2$ and $a_{13}   = a_{23} =0$ yields  $\k f_1 + \k f_2  =  \k g_1 + \k g_2.$ 
\end{proof}

The next lemma implies that any two distinct conics in a pencil of affine conics  are defined by independent quadratics.

\begin{lemma} 
\label{singular}
Let $f_1,f_2$ be independent quadratics.  
 
  \begin{itemize}
\item[$(1)$] 
Two  quadratics $g_1,g_2 \in \k f_1 + \k f_2$ are dependent if and only if one of  $g_1,g_2 $ is a scalar multiple of the other.


\item[$(2)$]  Two reducible quadratics in $\k f_1 + \k f_2+ \k$ are dependent if and only if 
they 
define 
   the same pair of non-parallel lines or they 
 define  pairs of parallel lines having the same midline. 
   \end{itemize}
   \end{lemma}

   \begin{proof}         For statement (1),  write $g_i=\alpha_i f_1 + \beta_i f_2$, where $\alpha_i,\beta_i \in \k$ are not both $0$.    Suppose $g_1$ and $g_2$ are dependent. Then  there are $\alpha,\beta \in \k$, not both $0$, such that $\deg(\alpha g_1 + \beta g_2)<2$. Since $f_1$ and $f_2$ are independent and
   $$\alpha g_1 + \beta g_2 = (\alpha \alpha_1 + \beta \alpha_2)f_1 +(\alpha \beta_1+\beta \beta_2)f_2,$$ it follows that $\alpha \alpha_1 + \beta \alpha_2 = \alpha \beta_1+\beta \beta_2=0$. Since at least one of $\alpha$ and $\beta$ is not zero,  $\alpha_1\beta_2 - \alpha_2\beta_1 =0$, and hence 
    the vectors $(\alpha_1,\beta_1)$ and $(\alpha_2,\beta_2)$ in $\k^2$ are linearly dependent. Consequently, $g_1 $ and $g_2$ are scalar multiples of each other. The converse of (1) is clear.

  To verify (2), suppose $g_1$ and $g_2$ are  reducible quadratics in $\k f_1 + \k f_2+\k.$  
There are $\lambda_1,\lambda_2 \in \k$ such that $g_1+\lambda_1,g_2 +\lambda_2 \in \k f_1 + \k f_2$.   Assume $g_1$ and $g_2$ are dependent.  Then so are $g_1 + \lambda_1$ and $g_2+\lambda_2$, and 
   by (1),   there is $0\ne \gamma \in \k$ such that $\gamma(g_1+\lambda_1) = g_2+\lambda_2$, so that $\gamma g_1$ is a degeneration  of $g_2$.  
    By Proposition~\ref{late lemma}, no degenerate hyperbola is the degeneration of a different degenerate hyperbola, so 
    if $g_1$ and $g_2$ are hyperbolas, then   $g_1$ and 
   $g_2$  define the same pairs of non-parallel lines.  If instead $g_1$ and $g_2$ define pairs of parallel lines, then by Proposition~\ref{late lemma}, as degenerations of each other, $g_1$ and $g_2$ must have zero sets that are pairs of parallel lines that share the same midline. 
  The converse of (2) also follows from Proposition~\ref{late lemma}. 
 %
   \end{proof}

In order to show that asymptotic pencils are nonempty, 
it suffices to prove  that pencils of affine conics contain at least one hyperbola. 
%

\begin{proposition} \label{generators2}
Every pencil of affine conics contains at least one hyperbola, and  
if $|\k|>3$ the pencil contains at least two  hyperbolas. 
\end{proposition} 

\begin{proof} 
Let
 $f_1,f_2$ be  independent quadratics. For each $i =1,2$, write $$f_i(X,Y) = a_iX^2+b_iXY+c_iY^2+ d_iX + e_iY+g_i,$$
  where the coefficients of $f_i$ are from $\k$. 
  Since $f_1$ and $f_2$ are independent, the vectors $(a_1,b_1,c_1)$ and $(a_2,b_2,c_2)$ in $\k^3$ are linearly independent. 
   The matrix having these two vectors as its rows has rank 2, and 
 so  its reduced row echelon form is 
 one of the following matrices,  where $b,c,c' \in \k$.
 \begin{center}
 $\begin{bmatrix}
 1 & 0 & c \\
 0 & 1 & c' \\
 \end{bmatrix} 
 \quad
 \begin{bmatrix}
 1 & b & 0 \\
 0 & 0 & 1 \\
 \end{bmatrix} 
 \quad
 \begin{bmatrix}
 0 & 1& 0 \\
 0 & 0 & 1 \\
 \end{bmatrix}. $
 \end{center}  
By switching  $f_1$ with $f_2$ and/or interchanging the variables $X$ and $Y$,  the only  cases that need to be considered are the first matrix or the second matrix with $b=0$. The latter case allows  a reduction to  
\begin{eqnarray*} 
f_1(X,Y) &= & X^2 + d_1X+e_1Y+g_1 \\
f_2(X,Y) & = &  Y^2+d_2X+e_2Y+g_1,
\end{eqnarray*}
and so $f_1 -f_2$ is a hyperbola in $\k f_1 + \k f_2$ (it has two points at infinity) and, if $|\k|>3$, then  we can choose $t \in \k$ such that $t^2 \ne 1$, so that 
$f_1-t^2f_2$ is a  hyperbola in $\k f_1 + \k f_2$ that is independent from $f_1-f_2$.

In the former case, that of the first matrix in reduced row echelon form, 
%
 we can  reduce to the situation in which  
the quadratics $f_1$ and $f_2$ are  
 \begin{eqnarray*} \label{case 1a} f_1(X,Y) &= & X^2 + c_1Y^2+ d_1X+e_1Y+g_1 \\ \label{case 1b}
  f_2(X,Y) &= & 
  XY+c_2Y^2+ d_2X+e_2Y+g_2.
  \end{eqnarray*}
  First observe that
$f_2$ is a hyperbola with points at infinity $[-c_2:1:0]$ and $[1:0:0]$, so the pencil $\k f_1 + \k f_2$ contains at least one hyperbola. 
 
 Now suppose 
   $|\k| > 3$. There is $0 \ne r \in \k$ such that $r \ne -c_2$ and $r^2 +2c_2r -c_1 \ne 0$.  Let $$\beta = -\frac{r^2+c_1}{c_2+r} \:\: {\mbox{ and }} \:\: s=-\beta-r.$$ A  calculation shows the degree two homogeneous component of $f_1+\beta f_2 $ is $$ X^2+ c_1Y^2 +\beta (XY+c_2Y^2) = (X-rY)(X-sY)$$
Also, $r \ne s$ since by the choice of $r$,  $$r-s =
2r+\beta =
 2r -\frac{r^2+c_1}{c_2+r} = 
\frac{r^2+2c_2r-c_1 }{r+c_2} \ne 0.$$ Therefore, 
$f_1+\beta f_2$ is a hyperbola whose  distinct points at infinity are $[r:1:0]$ and $[s:1:0]$,  and so $f_1+\beta f_2$ and $f_2$ are independent  hyperbolas in $\k f_1 + \k f_2$ since they do not share the same points at infinity.   
%
 %
%
%
%
\end{proof}

\begin{corollary} \label{generators2 cor} Every asymptotic pencil contains a degenerate hyperbola, and if 
 $|\k| > 3$ 
 every asymptotic pencil contains at least two degenerate  hyperbolas.\end{corollary}

\begin{proof} By Proposition~\ref{generators2}, every pencil of affine conics contains a hyperbola, and hence every asymptotic pencil contains   a degenerate hyperbola.  Suppose $|\k|>3$.
By Proposition~\ref{generators2}, each  pencil of affine conics contains at least two  independent hyperbolas. If the asymptotes of these two hyperbolas are the same pairs of lines, 
then  by Proposition~\ref{late lemma}, these asymptotes are degenerations of the hyperbolas, and so the equations for the hyperbolas differ by a constant, a contradiction to the fact that the hyperbolas are independent. Thus the asymptotic pencil contains at least two degenerate hyperbolas. \end{proof}

Proposition~\ref{generators2} and Corollary~\ref{generators2 cor} are not true if $\k$ has only 3 elements:

\begin{example} Let $\k$ be the field with 3 elements. The pencil generated by the parabola $X^2+Y$ and the degenerate hyperbola $XY+Y^2$ has, up to scalar multiple, only two other conics, the parabola $X^2+Y+XY+Y^2=(Y+2X)^2+Y$ and the   ellipse $X^2+Y+2XY+2Y^2$, whose zero set consists of the four points $(0,0),(0,1),(1,1),(1,2)$. By Proposition~\ref{late lemma}, the asymptotic pencil of $X^2+Y$ and  $XY+Y^2$ contains only one conic up to scalar multiple,  the degenerate  hyperbola $XY+Y^2$. 
\end{example}

Viewing two polynomials in $\k[X,Y]$ as {equivalent} if one is a nonzero scalar multiple of the other,  
the space of equivalence classes of  polynomials of degree at most $2$ is  five-dimensional projective space $\P^5(\k)$, where an equivalence class of a polynomial $aX^2+bXY+cY^2+dX+eY+f$ is represented by a point $[a:b:c:d:e:f] \in \P^5(\k)$.  Pencils of affine conics correspond to  lines under this representation. Asymptotic pencils are more complicated.

\begin{proposition} \label{cubic} If $\k$ is an algebraically closed field, then an asymptotic pencil is a planar cubic in  $\P^5(\k)$. 
\end{proposition} 

\begin{proof} 
 Let $f_1,f_2$ be  independent quadratics. 
 We show the asymptotic pencil generated by $f_1,f_2$ is a planar cubic. 
 For each $i =1,2$, write $$f_i(X,Y) = a_iX^2+2b_iXY+c_iY^2+ 2d_iX + 2e_iY+g_i,$$
  where the coefficients of $f_i$ are from $\k$. 
    Let $\Delta(T,U,V)$ be the determinant of the following matrix, where $T,U,V$ are indeterminates:
$$\begin{bmatrix} 
a_1U + a_2V & b_1 U+b_2V &   d_1U+ d_2V \\
 b_1U+ b_2 V&  c_1U+ c_2V &  e_1 U+ e_2V \\  d_1U+ d_2V &  e_1U+ e_2V & g_1U+ g_2V+T \\
\end{bmatrix}.$$
Let $\lambda,\alpha,\beta \in \k$, not all zero. Since $\k$ is algebraically closed, the quadratic  $\alpha f_1 + \beta f_2+\lambda$ is  reducible   over 
$\k$  if and only if   $\Delta(\lambda,\alpha,\beta)=0.$ 
Viewing $\Delta$ as a polynomial   over the ring $\k[U,V]$, the coefficient of $T$  is
$$\Phi(U,V) = (a_1c_1-b_1^2)U^2+(a_1c_2+a_2c_1-2b_1b_2)UV+(a_2c_2-b_2^2)V^2.$$ 
Since the degree of $T$ in $\Delta$ is at most $1$, any choice of point $[\alpha:\beta]$ on the projective line  $\P^1({\k})$ over ${\k}$  that is not a zero of $\Phi$ results in a choice for $T$ that is a root of $\Delta(T,\alpha,\beta)$, and hence yields a reducible quadratic in the asymptotic pencil generated by $f_1,f_2$.  

We claim  the polynomial $\Phi$ is not uniformly zero. Suppose to the contrary that it is. 
For each $i$, since $f_i$ is a quadratic, it cannot be that all of $a_i,b_i,c_i$ are zero. 
Suppose $a_1 =0$. From the coefficient of $U^2$ in $\Phi$ we obtain $b_1=0$ and hence $c_1\ne 0$.  Examination of the coefficient of $UV$ shows then that $a_2 = 0$, and hence, using the coefficient of $V^2$, $b_2 =0$, in which case $c_2 \ne 0$. In summary, $a_1 =b_1 =a_2=b_2=0$, $c_1 \ne 0$ and $c_2 \ne 0$.  Thus, with $\alpha = c_2$ and $\beta =-c_1$, the polynomial $\alpha f_1 + \beta f_2$ has degree at most $1$, contrary to the fact that $f_1$ and $f_2$ are independent. This contradiction implies that if $\Phi$ is uniformly $0$, then  $a_1 \ne 0$. A symmetrical argument shows $a_2 \ne 0$. 

Still assuming that $\Phi$ is uniformly $0$, the fact that $a_1 \ne 0$ and $a_2 \ne 0$ implies $c_1 = b_1^2/a_1$ and $c_2 = b_2^2/a_2$. Substituting this into the coefficient of $UV$, we conclude   $a_1b_2-a_2b_1 =0$. Since $a_1$ and $a_2$ are nonzero, the elements $b_1$ and $b_2$ are nonzero also. From the coefficients of $U^2$ and $V^2$, we obtain $b_1/a_1 = c_1/b_1$ and $b_2/a_2 = c_2/b_2$. From this and the fact that  $a_1b_2-a_2b_1=0$ it follows that $c_1/b_1 = c_2/b_2$, which implies  $b_1c_2-c_1b_2 =0$. Thus $a_1b_2-a_2b_1 =b_1c_2-c_1b_2=0$, and so the vectors $(a_1,2b_1,c_1)$ and $(a_2,2b_2,c_2)$ in $\k^3$ are linearly dependent. This implies there are $\alpha, \beta \in \k$, not both $0$, for which   $\alpha f_1 + \beta f_2$ has degree at most $1$, a contradiction that implies $\Phi$ is not uniformly $0$.

The asymptotic pencil generated by $f_1$ and $f_2$ is the set of quadratics $\alpha f_1 + \beta f_2 + \lambda$ such that $[\alpha:\beta:\lambda] \in \P^2(\k)$ and $\Delta(\lambda,\alpha,\beta) =0$. The polynomial $\Delta$ is a cubic since the polynomial $\Phi$ is not uniformly zero.  
The dual points of the quadratics in $\k f_1 + \k f_2 + \k$ lie in a plane in  $\P^5(\k)$, and so  the asymptotic pencil genereated by $f_1$ and $f_2$ is a cubic in this plane.  
\end{proof} 

If $\k$ is not an algebraically closed field, then by extending $\k$ to its algebraic closure, it follows that  an asymptotic pencil over $\k$ is a piece of the cubic   from Proposition~\ref{cubic} over the algebraic closure of $\k$.
 The extra points on the cubic that are not on the asymptotic pencil over $\k$ are the dual points of the ellipses in the original pencil of quadratics.  
The cubic in Proposition~\ref{cubic} is studied in more detail in \cite{OW6}.



\section{Pencils generated by quadrilaterals}

Classically, pencils of conics arise in connection with vertices of complete quadrilaterals (or quadrangles) in the projective plane because of the fact that the set of projective conics through any four distinct points in general position is a pencil; see Figures 1(c) and 2(b). 
However, there are  pencils that are not defined by the vertices of quadrilaterals in the affine plane, as in  Figures~1(a), 1(b), 2(a) and 2(c), and so 
there are more pencils than quadrilaterals. By contrast,  every nontrivial asymptotic pencil---nontrivial in the sense defined below---can be viewed as given by a quadrilateral.  The main purpose of this section is to prove this assertion.  

By a {\it $($complete$)$ quadrilateral} $Q=ABA'B'$ we mean two pairs of lines $A,A'$ and $B,B'$  in the affine plane  such that neither pair is a translation of the other,  these pairs do not share a line, the lines $A,A',B,B'$ do not all share a point, and these four lines are not all parallel.
The pairs $A,A'$ and $B,B'$  are the pairs of {\it opposite sides} of $Q$; the other pairs are {\it adjacent} sides of $Q$.
%
   The intersection of two adjacent   sides of $Q$   is a {\it vertex}. (If a pair of adjacent sides has parallel lines, then the vertex is the point at infinity for these sides.) The lines joining the pairs of opposite vertices  are the {\it diagonals} of $Q$. 
The definition  implies 
 a quadrilateral $Q$   has more than one vertex, that no adjacent sides are equal and  that the line at infinity is neither a  diagonal or a side of $Q$.  
Opposite sides can be equal, in which case $Q$ is a {\it degenerate} quadrilateral.  


\begin{definition}
The {\it pencil of $Q$} is the pencil defined by the reducible quadratics whose zero sets are the pairs of opposites sides of $Q$.   

\end{definition}



Proposition~\ref{quad exists} will show that as long as an asymptotic pencil is not trivial, in the following sense, there is a quadrilateral that generates it.

\begin{definition} 
An asymptotic pencil   is {\it trivial} if it consists only of degenerate hyperbolas, all of which have  the same center. \end{definition}


\begin{proposition} \label{trivial} If $\k$ is an algebraically closed field, then no asymptotic pencil is trivial. 
\end{proposition}

\begin{proof}  
We prove first that if $f_1$ and $f_2$ are independent  quadratics defining degenerate hyperbolas sharing the same center, 
 then the pencil generated by $f_1$ and $f_2$ contains a double line. 
 Since the zero sets of $f_1$ and $f_2$ each have two points at infinity,  we can assume 
 after an affine transformation that
 \begin{center} $f_1(X,Y)=XY$ and  $f_2(X,Y) = (aX+bY)(cX+dY)$, where $a,b,c,d \in \k$. 
 \end{center}
 Since  $f_2$ is not a double line and  $f_1$ and $f_2$ are independent,  all of the following conditions must hold:
 \begin{equation} \label{cons}
 ad-bc \ne 0 \quad a \ne 0 {\mbox { \rm or } } d\ne 0 \quad  b\ne 0 {\mbox { \rm or } } c \ne 0 \quad  a\ne 0 {\mbox { \rm or } } b\ne 0 \quad c\ne 0 {\mbox { \rm or } } d\ne 0. 
 \end{equation}  
To prove  
there is a double line in the pencil generated by $f_1$ and $f_2$, we show 
   there are pairs $\alpha,\beta$ and $e,f $ in $ \k$ such that at least one member of each pair is nonzero and
\begin{equation}
 \label{double} \alpha XY+\beta (aX+bY)(cX+dY) =  (eX+fY)^2.
 \end{equation}  
Expanding and collecting like terms, it follows that such a double line $(eX+fY)^2$ can be found in the pencil if and only if 
$$\beta ac -e^2 = 0 \quad \alpha+\beta (ad+bc)  -2ef =0 \quad \beta bd -f^2=0.$$

If $b=0$, then from (\ref{cons}) we have  $a \ne 0$, $c \ne 0$ and $d \ne 0$, and so with 
$$\alpha = -\frac{ad+bc}{ac} \quad \beta = \frac{1}{ac} \quad e = 1 \quad f = 0$$ 
 we obtain equation (\ref{double}).
 Similarly, if $d =0$, we may choose $\alpha,\beta,e,f$ to obtain equation (\ref{double}).
 
 Now suppose $b \ne 0$ and $d \ne 0$. Since $\k$ is algebraically closed, there is  $\theta \in \k$ such that  $bd\theta^2=ac$. With the assignments 
$$\alpha = \frac{2bd\theta-ad-bc}{bd} \quad \beta = \frac{1}{bd} \quad e = \theta \quad f = 1$$ we obtain equation (\ref{double}). Therefore, in all cases the asymptotic pencil generated by  $f_1$ and $f_2$ contains a double line. 

The  proposition now follows since  Corollary~\ref{generators2 cor} and the fact that $\k$ is infinite imply  there are at least two independent degenerate hyperbolas in any asymptotic pencil over $\k$. If these two hyperbolas do not have the same center, then the asymptotic pencil is nontrivial. If they 
do
have the same center, then, by what we have established, the pencil they generate contains a double line, and so the asymptotic  pencil is not trivial since it contains this double line also.  
\end{proof} 

It is easy to construct examples of trivial asymptotic pencils over non-algebraically closed fields. For example, if $\k$ is the field of real numbers, the asymptotic pencil for the conics $XY$ and $X^2-Y^2$ is trivial, as can be checked directly or using the calculations in the proof of Proposition~\ref{trivial}. A more general explanation for this example is that the two degenerate hyperbolas that define the asymptotic pencil are interleaved, with each linear component of each hyperbola  between the two linear components of the other hyperbola. Whenever this is the case, the two degenerate hyperbolas sharing the same center will define a trivial asymptotic pencil. 
We will be interested in nontrivial asymptotic pencils, there being little to say about the trivial case.

The asymptotic pencil in Figure 2(c) illustrates the next lemma. In that figure, the black line is shared by every hyperbola  in the asymptotic pencil.





\begin{lemma} \label{trap} The following are equivalent for an asymptotic pencil $\A$. 
\begin{itemize}
 \item[$(1)$]  $\A$ contains two degenerate hyperbolas that intersect in a line. 
\item[$(2)$] 
 $\A$ contains a degenerate  parabola and a degenerate  hyperbola that intersect in a line.
\item[$(3)$]  $\A$ contains two  conics that intersect in a line.
\end{itemize} 
In this case, all  hyperbolas in $\A$  intersect in the same line. 
%

%
%

\end{lemma}

\begin{proof} 
We will prove the last assertion of the lemma in the course of proving (1) implies~(2). 

%
(1) $\Rightarrow$ (2): 
Suppose $g_1$ and $g_2$ define degenerate hyperbolas in $\A$ that share a line. 
After an affine transformation, we may assume  $g_1(X,Y) = XY$ and  the shared line is $X=0$. Write $g_2(X,Y) = X(dX+eY+f)$, where $d,e,f\in \k$ and $d$ and $e$ are not both $0$. Since $g_2$ is a hyperbola, 
 $e \ne 0$.
 A quadratic in $\k g_1 + \k g_2 + \k$ has the form 
  $$\alpha XY+\beta X (dX+eY+f)+\lambda = X(\beta d X +(\beta e+\alpha)Y+\beta f) +\lambda,$$ 
 where $ \alpha,\beta,\lambda \in \k$  and $ \alpha,\beta$ are not both $0$.
  By Proposition~\ref{late lemma}, a necessary condition for such a quadratic to be reducible is that  $\lambda =0$ or $\beta e +\alpha =0$. 
   If this quadratic is a hyperbola, then  $\beta e + \alpha \ne 0$ and $\lambda=0$, and so we  have the degenerate hyperbola $X(\beta d X +(\beta e+\alpha)Y+\beta f)$, which shares the linear component~$X$ with $g_1$ and $g_2$. 
Note  also that $\A$ contains the degenerate parabola $$-eXY + X(dX+eY+f) = X(dX+f),$$ and hence by Proposition~\ref{late lemma} contains a degenerate parabola that shares a linear component with the hyperbola $XY$ in $\A$. This proves that if $\A$ contains two independent hyperbolas that share a linear component, then $\A$ contains a degenerate parabola and a degenerate hyperbola that intersect in a line, and  all hyperbolas in $\A$ intersect in this same line.

(2) $\Rightarrow$ (1): 
Suppose $\A$ contains
a  degenerate hyperbola $g_1=0$ and a degenerate parabola $g_2=0$ that intersect in the same line. Without loss of generality, \begin{center}
$g_1(X,Y) = XY$ and $g_2(X,Y) = X(X-a)$ for some $a \in \k$. 
\end{center} 
Thus
  $\A$ contains the   hyperbola  $$  XY+  X(X-a)  = 
  X(X +  Y-a).$$ 
This hyperbola and the hyperbola defined by $g_1$ are distinct   and intersect in a line. 


(2) $\Leftrightarrow$ (3):
It is clear (2) implies (3). Conversely, suppose  $g_1$ and $g_2$ are independent   quadratics whose zero sets are in $\A$  and share a linear component. If $g_1$ and $g_2$ are both hyperbolas, then 
since we have shown already that (1) and (2) are equivalent, statement (2) follows. So suppose without loss of generality that $g_1$ is not a hyperbola. Then, since $g_1$ is reducible, $g_1$ is a  parabola. If $g_2$ is also a   parabola, then since $g_1$ and $g_2$ share a linear component, all the linear components of $g_1$ and $g_2$ are parallel, contrary to the assumption that $g_1$ and $g_2$ are independent. Thus $g_2$ is a   hyperbola, and so $\A$ contains a   hyperbola and a   parabola that share a line. 
%
\end{proof}

 By Corollary~\ref{generators2 cor}, if $|\k|>3$, then every asymptotic pencil can be specified by two degenerate hyperbolas, in fact by  any two different reducible quadratics in the asymptotic pencil. 
 If these reducible conics intersect in a line, they do   not define  a pair of opposite sides of a  quadrilateral. 
However, as we show next, as long as the asymptotic pencil is nontrivial,   these reducible conics can be traded for two more that make a quadrilateral. This is true even if $|\k|= 3$ elements, but to obtain a non-degenerate quadrilateral we need once again that $|\k|>3$.

\begin{proposition} \label{quad exists} 
An asymptotic pencil  ${\mathbb{A}}$ is nontrivial if and only if ${\mathbb{A}}$ is the asymptotic pencil of   some quadrilateral $Q$.     If $|\k|>3$, then  $Q$ can be chosen a non-degenerate quadrilateral. 

\end{proposition} 

\begin{proof} 
If $\A$ is the asymptotic pencil of a quadrilateral $Q$, the pairs of opposite sides of $Q$ define reducible conics in $\A$ that, if they are hyperbolas, do not share the same center. This guarantees $\A$ is nontrivial.   To prove the converse, first note that  
by Corollary~\ref{generators2 cor} there is a degenerate hyperbola $f_1$ in $\A$.  
Since $\A$ is nontrivial,  $\A$ contains 
either a hyperbola with a different center than that of $f_1$ or a degenerate parabola.

If $\A$ contains a degenerate parabola $f_2$, then  
  by Proposition~\ref{late lemma} and the fact that $|\k| \geq 3$, the asymptotic pencil $\A$ contains a double line and a  degenerate parabola   whose linear components are distinct and parallel to the double line. One of these two degenerate parabolas does not share a linear component with $f_1$ and hence forms a quadrilateral with the zero set of $f_1$. Thus if $\A$ contains a degenerate parabola, $\A$ contains a quadrilateral $Q$ and 
 $\A$ is the asymptotic pencil of $Q$. 
  
 If $\A$ does not contain a  parabola, then $\A$ contains a  hyperbola $f_2$ that does not share a center with $f_1$.  Since $\A$ does not contain a   parabola, Lemma~\ref{trap} implies $f_1$ and $f_2$ do not share a linear component, and hence $f_1$ and $f_2$ define a quadrilateral $Q$. 
 Lemma~\ref{singular} implies $f_1$ and $f_2$ are independent, so 
 $\A $ is the asymptotic pencil of $Q$.  This proves the first assertion of the proposition.

 Finally, suppose 
     $|\k| >3$. If $Q$ is degenerate, then $Q$ has a pair of parallel opposite sides that are equal.   By Proposition~\ref{late lemma} and the fact that $|\k|>3$,  $\A$ contains a double line and two  degenerate parabolas whose zero sets are  distinct and  whose linear components are parallel to the double line. We can replace the double line that is a pair of opposite sides of $Q$ with one of the degenerate parabolas that does not share a linear component with the other conic that defines $Q$. In doing so, we obtain a non-degenerate quadrilateral in $\A$.  
\end{proof}




\section{Bisectors of conics}

For a line $\ell$, denote by $\overline{\ell} = \ell \cup \{\infty\}$ the projective closure of  $\ell$, i.e., the line $\ell$ and its point at infinity $\infty$ in the projective plane. 
We define the {\it midpoint} of  two points $p,q$ on $\overline{\ell}$ as the usual midpoint of two points if $p,q\in \ell$, and  as $\infty$ if $p \in \ell$ and $q =\infty$. In the first case the midpoint is {\it finite} and in the second it is {\it infinite}. 
If  $p = q = \infty$,  the midpoint is   {\it undetermined}. 





It will be convenient to have  terminology for the different ways a line can intersect a conic.  
A line $\ell$ {\it meets} a conic defined by a quadratic $f$ if there is a point, possibly at infinity, on both $\ell$ and $f$; i.e., 
 the projective closure of $\ell$ has nonempty intersection with the   projective closure of the zero set of  $f$. 
 The line $\ell$
 {\it crosses} the zero set $V(f)$ (or $f$, for short)   if $\ell$ is not a component of $V(f)$ and  
 $\ell$ meets $V(f)$ in at least one point in the affine plane. Note that any line meets any reducible conic but may not cross it, since the line could be a component of the conic or meet the conic only at infinity.  
If $\ell$ crosses $f$, then   $\mid_{f}(\ell)$ denotes the midpoint of the two crossing points, one of which may be at infinity.


   
   
   

\begin{definition} 
A line $\ell$ {\it bisects} a collection ${\mathcal{C}}$ of conics (or the quadratics that define the conics) with {\it midpoint} $m$  if 
$m =\mid_f(\ell) $ for all quadratics $f$  whose zero sets are in ${\mathcal{C}}$ and are crossed by $\ell$. 
  If no conic in ${\mathcal{C}}$ is crossed by $\ell$, then $\ell$ is vacuously a bisector of ${\mathcal{C}}$ and its midpoint $\mid_f(\ell)$ is {\it undetermined} for each quadratic $f$ whose zero set is in ${\mathcal{C}}$.

\end{definition}

    \begin{figure}[h] 
     \label{9noobconics}
 \begin{center}
\includegraphics[width=.55\textwidth,scale=.09]{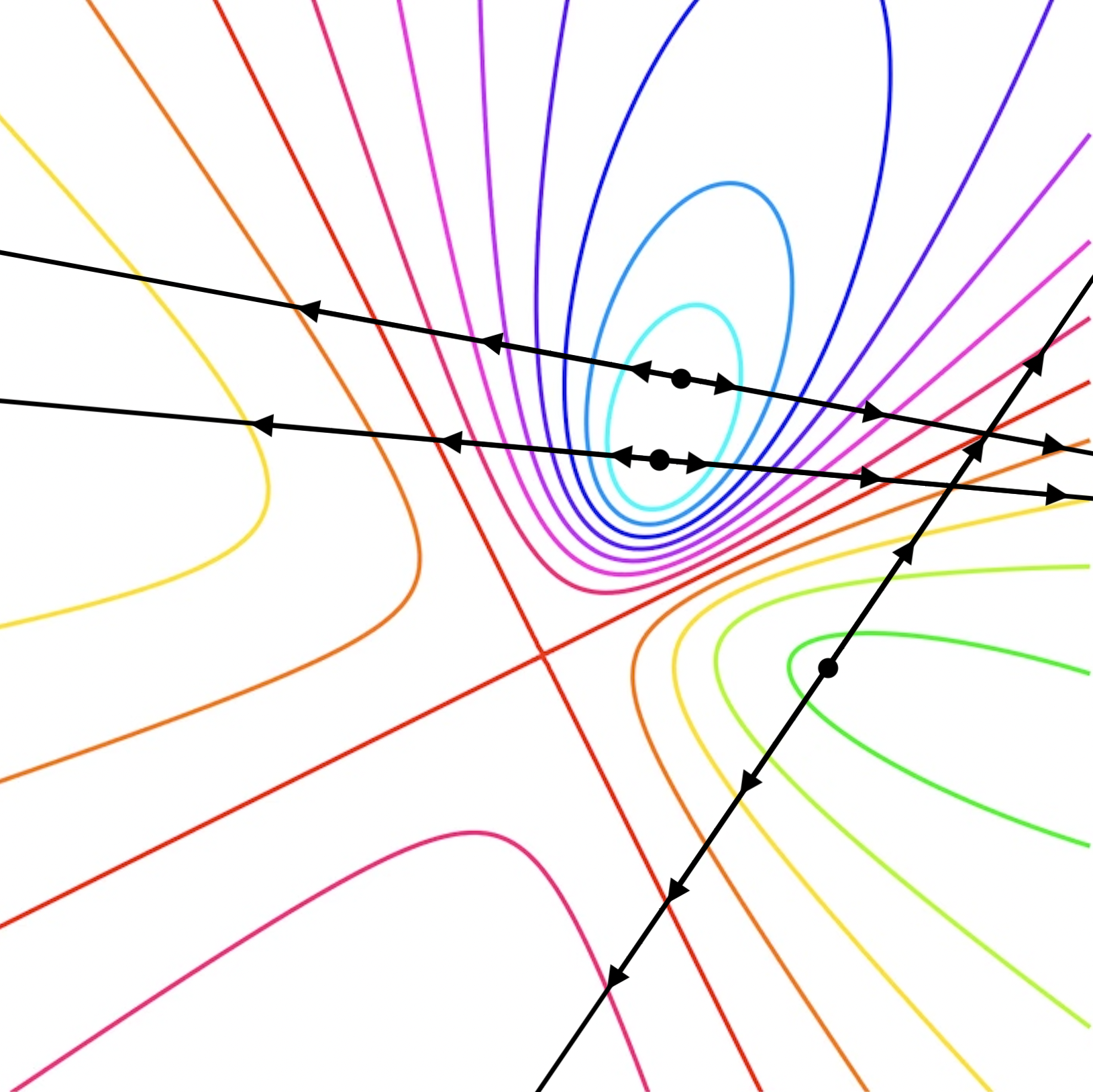} 
 \end{center}
 \caption{The three black lines are bisectors of the  collection of conics in the figure. The midpoints are indicated by black dots and the arrows are to help visualize the bisection property.}
\end{figure}

Figure 3 illustrates the definition. 
The next lemma provides the first step in  relating asymptotic pencils to the bisection property. 

 \begin{lemma}  \label{main lemma} Let ${\mathbb{A}}$ be the asymptotic pencil generated by  quadratics $f_1$ and $f_2$. 
A
 line $\ell$  that meets  the zero sets of $f_1$ and $f_2$ bisects the pair $\{f_1,f_2\}$ if and only if $\ell$    is a   component of a reducible conic    in ${\mathbb{A}}$. 
 %
 
 
\end{lemma}

\begin{proof}  
After an affine transformation, we may assume $\ell$ is the line $Y=0$.  The line $\ell$ is a member of a pair in ${\mathbb{A}}$
 if and only if there are $\alpha,\beta,\lambda,t,u,v \in \k$ such that  $\alpha$ and $\beta$ are not both $0$; $t$ and $u$ are not both $0$; and  
\begin{eqnarray}
\label{aaa} \alpha f_1+
\beta f_2 = Y(tX-uY+v)+\lambda.
\end{eqnarray}
For each $i$, write $$f_i(X,Y) = a_iX^2+b_iXY+c_iY^2+d_iX+e_iY+g_i.$$ 
Expanding, collecting like terms and equating coefficients, we obtain that equation (\ref{aaa}) holds if and only if 
\begin{eqnarray} \label{tuv}
t & = &  \alpha b_1+\beta b_2 \\
 \label{tuv2}
u &= & -\alpha c_1-\beta c_2  \\ 
 \label{tuv3}
v & = & \alpha e_1+\beta e_2  \\
 \label{tuv4}
\lambda & = &  \alpha g_1+\beta g_2 \\
\label{former}
0 & = & \alpha a_1 +\beta a_2 \\
\label{latter}
0 & = &  \alpha d_1+\beta d_2.
\end{eqnarray}
Therefore, the line $\ell$ is a member of a pair in ${\mathbb{A}}$
 if and only if there are $\alpha,
 \beta,\lambda,t,u,v \in \k$ such that  $\alpha$ and $\beta$ are not both $0$; $t$ and $u$ are not both $0$; and equations (\ref{tuv})--(\ref{latter}) hold. 
 
We prove the lemma by establishing a series of simple claims.
 
 \smallskip

{\textsc{Claim 1:}}  
 $\ell$ is a member of a pair in ${\mathbb{A}}$ if and only if $a_1d_2 - a_2d_1 = 0$.
 
 \smallskip

There is not a solution to equations (\ref{tuv})--(\ref{latter}) in which  $t=u=0$ but  $\alpha $ and $\beta$ are nonzero,  
 since otherwise
equations (\ref{tuv}), (\ref{tuv2}) and (\ref{former}) with $t =u=0$ imply
$\alpha f_1 + \beta f_2 $ is a polynomial of degree at most $1$, contrary to the fact that $f_1$ and $f_2$ are independent. 
 Using this observation and equations (\ref{former}) and (\ref{latter}), we conclude 
that  
 $\ell$ is a member of a pair in ${\mathbb{A}}$
 if and only if 
$a_1d_2 - a_2d_1 = 0$.


\smallskip

{\textsc{Claim 2:}} If $\mid_{f_1}(\ell)$ and $\mid_{f_2}(\ell)$ are both finite, 
then $\ell$ is a member of a pair in  ${\mathbb{A}}$ if and only if $\ell$ bisects $\{f_1,f_2\}$. 

\smallskip

Suppose $\mid_{f_1}(\ell)$ and $\mid_{f_2}(\ell)$ are both finite. Then $\ell$ crosses both $f_1$ and $f_2$ and does not meet either at infinity. Thus for each $i =1,2$, we have  $a_i \ne 0$ since otherwise the line $\ell$, which is given by  $Y=0$, meets $f_i$ at a point at infinity. 
The points of intersection of $\ell$ with $f_i$ are $(\sigma_i,0)$ and $(\tau_i,0)$, where $\sigma_i$ and $\tau_i$ are the zeroes of $a_iX^2+d_iX+g_i$. 
Necessarily, $\sigma_i + \tau_i = -d_ia_i^{-1}$. 
Therefore, $\sigma_1 + \tau_1 = \sigma_2 + \tau_2$ if and only if  
$d_1a_2 -d_2a_1 =0$. 
The midpoint of the points where $\ell$ crosses $f_i$ is $ (\frac{1}{2}(\sigma_i+\tau_i),0)$ and so $\ell$ bisects the pair $\{f_1,f_2\}$ if and only if $d_1a_2 -d_2a_1 =0$. By Claim~1, this is the case if and only if   $\ell$ is a member of a pair in  ${\mathbb{A}}$.

  \smallskip
  
  {\textsc{Claim 3:}}  If at least one of $\mid_{f_1}(\ell)$ or $\mid_{f_2}(\ell)$  is undetermined, then  $\ell$ is a member of a pair in  ${\mathbb{A}}$ and $\ell$  bisects $\{f_1,f_2\}$. 


  \smallskip
  
  Without loss of generality, $\mid_{f_1}(\ell)$  is undetermined and so 
  $\ell$ does not cross $f_1$.  By assumption, $\ell$ meets $f_1$, so either $\ell$ is a component of $f_1$ or $\ell$ meets $f_1$ at infinity only. In either case, $a_1=d_1=0$, so 
%
 by Claim 1, $\ell$ is a member of a pair in $\A$, and by definition $\ell$ bisects $\{f_1,f_2\}$ since   $\mid_{f_1}(\ell)$ is undetermined and hence  the position of  $\mid_{f_2}(\ell)$ does not matter for the sake of guaranteeing bisection of $\{f_1,f_2\}$.

    \smallskip
  
  {\textsc{Claim 4:}}  If one of 
 $\mid_{f_1}(\ell)$ and $\mid_{f_2}(\ell)$ is finite and the other is infinite, then $\ell$ is not a member of a pair in  ${\mathbb{A}}$ and $\ell$ does not bisect $\{f_1,f_2\}$. 
  
  \smallskip
  
  Since the two midpoints do not agree, $\ell$ does not bisect $\{f_1,f_2\}$. Without loss of generality, suppose $\mid_{f_1}(\ell)$ is infinite and $\mid_{f_2}(\ell)$ is finite. Then $a_1=0$ and $a_2 \ne 0$ since $\ell$ is the line $Y=0$ and $\ell$ meets $f_1$ at infinity but does not meet $f_2$ at infinity. Thus $a_1d_2 -d_1a_2 =0$ if and only if $d_1=0$. However, if $d_1 =0$, then $\mid_{f_1}(\ell)$ is undetermined since also $a_1 =0$, a contradiction. Thus Claim 1 implies $\ell$ is not a member of a pair in  ${\mathbb{A}}$

    \smallskip
  
  {\textsc{Claim 5:}}  If both  
 $\mid_{f_1}(\ell)$ and $\mid_{f_2}(\ell)$ are infinite, then $\ell$ is  a member of a pair in  ${\mathbb{A}}$ and $\ell$  bisects $\{f_1,f_2\}$. 
  
  \smallskip
 
 It is clear $\ell$  bisects $\{f_1,f_2\}$ since  $\mid_{f_1}(\ell) = \mid_{f_2}(\ell)$. Also, since both midpoints are infinite, we have $a_1=a_2=0$, which implies $a_1d_2-a_2d_1=0$, so that  $\ell$ is  a member of a pair in  ${\mathbb{A}}$ by Claim 1.
  \end{proof}






The next theorem shows that bisection of a single pair of independent quadratics lifts to bisection of an entire ``affine net'' of these two conics. 
The first two paragraphs of the proof of the  theorem can be avoided  by using a general version of Desargues' Involution Theorem, such as in \cite[14.2.8.3, p.~125]{Berger}, but instead we use our bisector methods to prove the theorem and then in Corollary~\ref{des} derive Desargues'  theorem as a consequence. 


\begin{theorem} \label{amazing}  Let $f_1$ and $f_2$ be independent quadratics, and let $\ell$ be a line that meets the zero sets of $f_1$ and $f_2$.  
If $\ell$ bisects $\{f_1,f_2\}$, then $\ell$ bisects the set of quadratics in $\k f_1 +\k f_2 + \k$ and hence also the
 asymptotic pencil generated by $f_1,f_2$. 
%
%
\end{theorem} 

\begin{proof}  Suppose $\ell$ bisects $\{f_1,f_2\}$. 
After an affine transformation, we may assume $\ell$ is the line $Y=0$.  Suppose first that 
$\mid_{f_1}(\ell)$ and $\mid_{f_2}(\ell)$ are undetermined. 
As in the proof of  Claim 3 of Lemma~\ref{main lemma},
 it follows that for $i=1,2$, 
 \begin{center}
 $f_i(X,Y) = b_iXY+c_iY^2+e_iY+h_i$ for some $b_i,c_i,e_i,h_i \in \k$. 
 \end{center}
  Every quadratic in  $\k f_1 + \k f_2 + \k$ thus has this form, and so every such quadratic  either has $Y=0$ as a component  or meets $Y=0$ at infinity only. 
 Thus $\ell$ does not cross any conic defined by a quadratric  in $ \k f_1 + \k f_2 + \k$, which implies 
 $\ell$ vacuously bisects every conic defined by a quadratic in   $\k f_1 + \k f_2 + \k$. 
 For the rest of the proof we assume without loss of generality that $\ell $ crosses $f_1$   and so  $\mid_{f_1}(\ell)$ is determined.

By Lemma~\ref{main lemma}, 
$\ell$ is a component of a conic  in  the asymptotic pencil generated by $f_1,f_2$. 
Let  $g$ be a quadratic in  $\k f_1 + \k f_2 + \k$   that defines a conic   $\ell$ crosses. It suffices to show $\ell$ bisects the set of the three conics defined by  $f_1,f_2,g$. 
  If $g$ and $f_1$ are independent, then  $f_1,f_2$ and $f_1,g$ generate the same asymptotic pencil  by Lemma~\ref{generators}, and since $\ell$ meets $g$ and $f_1$, 
and $\ell$ is a component of a conic in the asymptotic pencil defined by $f_1,g$,
 Lemma~\ref{main lemma} implies that $\ell$ bisects the pair of conics defined by  $f_1$ and $g$.  Therefore, since $\mid_{f_1}(\ell)$ is determined,  $\ell$ bisects $g$ with midpoint $\mid_g(\ell)=\mid_{f_1}(\ell)$.


It remains to examine the case in which  $f_1$ and $g$ are dependent, and to prove in this case that $\mid_{f_1}(\ell) = \mid_g(\ell)$. 
Since $g \in  \k f_1 + \k f_2 + \k$, there is $\lambda \in \k$ such that $g - \lambda \in \k f_1 + \k f_2$. Since $g- \lambda$ and $f$ are dependent and in $\k f_1 + \k f_2$, Lemma~\ref{singular}(1) implies $g - \lambda = \gamma f_1$ for some $0 \ne \gamma \in \k$.  Thus since $\mid_{f_1}(\ell) = \mid_{\gamma f_1}(\ell)$, it suffices to show that $\mid_g(\ell) 
 = \mid_{g-\lambda}(\ell)$.
 
 By assumption, $\ell$ crosses $g$, and since $\ell$ crosses $f_1$,   $\ell$ also crosses $g - \lambda=\gamma f_1$.  
Write
 $$g(X,Y) = aX^2+bXY+cY^2+dX+eY+h.$$
 Since the line $\ell$ is given by 
$Y=0$ and $\ell$ crosses $g$, the line  $\ell$ is not  a component of $g$ and $aX^2+dX+h$ has a root in $\k$. 
If $a =0$, then $\ell$ meets $g$ and $g - \ell$ at infinity, as well as in the affine plane, and so 
 $\mid_g(\ell) = \infty =  \mid_{g -\lambda}(\ell)$.
 If $a \ne 0$, then since $aX^2+dX+h$ has a root $\sigma$ in $\k$, there is another root $\tau \in \k$. The points $(\sigma,0)$ and $(\tau,0)$ are  where $\ell$ crosses $g$, and  
 $\sigma + \tau = -d/a$.  Similarly, $g-\lambda = aX^2+dX+h-\lambda$ has roots $\sigma',\tau' \in \k$ with $\sigma'+\tau' = -d/a$.  The line $\ell$ crosses $g-\lambda$ at $(\sigma',0)$ and $(\tau',0)$, so since $\sigma + \tau = \sigma' +\tau'$, we conclude $\mid_g(\ell) = -d/(2a)= \mid_{g -\lambda}(\ell)$, which proves the claim and completes the proof of the theorem.
\end{proof}

A line {\it bisects a quadrilateral} $Q$ if it bisects the set of pairs of opposite sides of $Q$. 

\begin{corollary} \label{amazing cor}
A line  bisects a quadrilateral $Q$  if and only if it bisects the set of conics defined by quadratics in $\k f_1 + \k f_2 + \k$, where $f_1 $ and $f_2$ are quadratics that define the pairs of opposite sides of $Q$. 

\end{corollary}
 
 \begin{proof}  The quadratics $f_1$ and $f_2$ are independent since $Q$ is a quadrilateral, so we  
apply Theorem~\ref{amazing}. 
 \end{proof}

\begin{corollary} \label{amazing cor2} Let $Q$ be a quadrilateral with four distinct vertices, one of which is possibly at infinity. 
If a  line  bisects  $Q$, it bisects the set of  conics through the vertices of $Q$.
\end{corollary} 

\begin{proof}  Since the set of conics through four distinct points in general position is a pencil, we may apply Corollary~\ref{amazing cor}. 
\end{proof}


\begin{corollary}   {\rm (Desargues' Involution Theorem)}  \label{des}
Let ${\mathcal{P}}$ be a pencil of conics in  $\P^2(\k)$, and let $\ell$ be a line in $\P^2(\k)$   that  does not pass through a basepoint of ${\mathcal{P}}$. 
 There is an involution $\Lambda$ on $\ell$ such that if $p$ and $q$ are the points of intersection of $\ell$ and a conic in ${\mathcal{P}}$, then $p$ and $q$ are conjugate under $\Lambda$. 
\end{corollary} 

\begin{proof} Let $p_1,q_1$ be the points of intersection of $\ell$ with a conic defined by $f_1 \in {\mathcal{P}}$, and let $p_2,q_2$ be the points of intersection of $\ell$ with a conic defined by $f_2 \in {\mathcal{P}}$ such that $\{p_1,q_1\} \ne \{p_2,q_2\}$. 
There is a projective transformation $T$ of the projective plane  that carries $\ell$ onto a line $T(\ell)$ such that in some affine chart containing $T(p_1),T(q_1),T(p_2),T(q_2)$, the midpoint of $T(p_1)$ and $T(q_1)$ is the same as the midpoint of $T(p_2)$ and $T(q_2)$. 
In this chart, the dehomogenizations of $f_1$ and $f_2$ are independent by 
 Lemma~\ref{singular}(1), and also in this chart,  
 the line $f(\ell)$  bisects the images of $f_1$ and $f_2$ under $T$.  
 By Corollary~\ref{amazing cor}, $T(\ell)$ bisects all the conics through the vertices of $Q$. Let $\Lambda$ be the reflection on $T(\ell)$ about the midpoint of $T(p_1)$ and $T(q_1)$ (which is also the midpoint of $T(p_2)$ and $T(q_2)$) that sends the point at infinity of $T(\ell)$ (with respect to the given chart) to itself. 
Then $T^{-1}\circ \Lambda \circ T$ is an involution on $\ell$ such that the pairs of points where the conics through the vertices of $Q$   cross $\ell$ are conjugate. 
\end{proof}

\section{Asymptotic pencils and bisector fields}

Since pairs of lines can be viewed as zero sets of  reducible conics, the definition of bisection of conics in the last section subsumes that of bisection of collections of line pairs, and for such collections we single out in the next definition the property that ``everything bisects everything.''


\begin{definition} \label{bis def} A  set $\B$  of  paired lines is a {\it bisector arrangement} if  each line $\ell$ in each pair in $\B$    bisects  $\B$. We denote by $\mid_\B(\ell)$ the midpoint of $\ell$ as a bisector of $\B$.  
 A bisector arrangement is {\it trivial}   if every pair is a translation of every other pair,  
all lines in the arrangement go through the same point, or all lines are parallel.  
A  {\it bisector field} is a nontrivial bisector arrangement that  
 cannot be extended to a larger bisector arrangement. 
\end{definition} 

As will follow from Theorem~\ref{main theorem 1}, a bisector arrangement need not be finite. The trivial nature of a trivial  bisector arrangement is that in these cases the bisector arrangement has the property that each line in the arrangement bisects every pair in the arrangement for uninteresting  reasons:
%
If all the lines in  a collection of pairs of lines meet at a point $p$ in the affine plane, then each line in the arrangement  bisects each  pair with midpoint $p$, while if all the lines in the  pairs are parallel and hence meet at a point at infinity, then each line vacuously bisects each pair with undetermined midpoint. 
If instead, each pair is a translation of every other pair, then no matter the placement of the pairs in the collection, each line in the arrangement  bisects the collection with  midpoint at infinity. 


The next lemma can be viewed as a  uniqueness statement for the pairings in a bisector arrangement.  The bisector configuration in Figure 2(c) shows the statement of the lemma is optimal. 

 \begin{lemma} \label{unique}

 Let  $\B $ be a nontrivial bisector arrangement consisting of two line pairs $\PP_1$ and $\PP_2$. If $\PP$ is a line pair such that $\B \cup \{\PP\}$ is a bisector arrangement extending   $\B$, then there is a line $\ell$ in $\PP$ such that the only line pair $\PP'$   containing $\ell$ for which $\B \cup \{\PP'\}$ is a bisector arrangement extending $\B$ is $\PP' = \PP$.

\end{lemma}

\begin{proof}  Let $\PP$ be a line pair for which $\B \cup \{\PP\}$ is a bisector arrangement. 
We will prove a series of claims about the statement 
\begin{itemize}
\item[($*$)] 
{\it There is a line $\ell$ in $\PP$ such that the only line pair $\PP'$   containing $\ell$ for which $\B \cup \{\PP'\}$ is a bisector arrangement is $\PP' = \PP$.}  
\end{itemize}
Write 
 $\PP_1 =\{A,A'\}$, $\PP_2=\{B,B'\}$ and $\PP=\{C,C'\}$.  
 
 \smallskip

{\textsc{Claim 1.}} At  most two of the lines in the arrangement $\B$ have
an infinite midpoint.

\smallskip

Suppose without loss of generality that
the bisectors $A,A',B$ have infinite midpoints in $\B$.  
Then $A$ and $A'$ are 
each parallel to exactly one of $B$ and $B'$, and $B$ is parallel to exactly one of $A$ and $A'$, so it follows that each line in each of the two pairs $\{A,A'\}$ and $\{B,B'\}$ is parallel to exactly one line in the other pair, contradicting the fact that $\B$ is a nontrivial bisector arrangement. Thus at most two of the lines in $\B$ have an infinite midpoint.

\smallskip

{\textsc{Claim 2.}} Statement ($*$) is valid if at least three bisectors   in $\B$  have a finite midpoint.

\smallskip



Without loss of generality, suppose $\mid_\B(A),\mid_\B( A'), \mid_\B(B)$ are all finite. 
Since $\B \cup \{ \PP\}$ is a bisector arrangement,  this implies
 $C$ and $C'$ cross each of $A,A',B$. We claim that at least two of the points $A \cdot C', A' \cdot C', B \cdot C'$ are not equal. Suppose otherwise. Then 
 $\mid_\B(B) = \mid_\B(C') = A \cdot A'$ since $B$ and $C'$ are bisectors in $\B$ and go through $A \cdot A'$.  
 The line $C'$ bisects the pair $B,B'$ with midpoint $\mid_\B(C')$, so since $B$ goes through this point, so does $B'$.  But then all four lines in $\B$ go through the point $\mid_\B(C')$, contradicting the fact that $\B$ is a nontrivial bisector arrangement. Thus at least  two of the points $A \cdot C', A' \cdot C', B \cdot C'$ are not equal.  We will show that only the line $C'$ can be paired with $C$ if $\B \cup \{\{C,C'\}\})$ is to be a bisector arrangement. 
 
Let $L,M \in \{A,A',B\}$ such that  $L \cdot C' \ne  M \cdot C'$.  Since $L$ bisects $\B\cup \{\{C,C'\}\}$ and $C$ and $C'$ cross $L$, the point $L \cdot C'$  where $C'$ crosses $L$ is  uniquely determined by $\mid_\B(L)$  and the point $L \cdot C$ where $C$ crosses $L$. Similarly, the point $M \cdot C'$ where $C'$ crosses $M$ is  uniquely determined by $\mid_\B(M)$  and the point $M \cdot C$ where $C$ crosses $M$. Since $L \cdot C' \ne M \cdot C'$, it follows that, given $C$, the line  $C'$ is the unique line for which $\B \cup \{\{C,C'\}\}$ is a bisector arrangement.

\smallskip

{\textsc{Claim 3.}} 
Statement ($*$) is valid if at  least one of the bisectors in   $\B$   has an infinite midpoint and at least one has a finite midpoint. 

\smallskip
Assume $\mid_\B(A) = \infty$. 
Since $A$ bisects $\B \cup \{\{C,C'\}\}$, 
 $A$ is parallel to at least one of $C$ and $C'$, say $C$. 
We will show $C$ is uniquely determined by  $C'$ and $\B$. 
Since $\mid_\B(A)$ is infinite, 
 $A$ is parallel to exactly one of $B$ and $B'$, say $B$, and so $A$, $B$ and $C$ are all parallel.  By {assumption}, at least one of the lines $A,A',B,B'$ has a finite midpoint in $\B$. 
 Since $A$, $B$ and $C$ are all parallel, the midpoints of these lines cannot be finite, and so at least one line $L \in \{A',B'\}$  
 has a finite midpoint. Let $M$ be the other line in this set. 
  The line $L$  is not parallel to 
  $C$ or $C'$ 
  since $\mid_\B(L) $ is finite and $L$ bisects $C,C'$ with this midpoint. 
Thus, 
since $L$ crosses both $C$ and $C'$, the point $C \cdot L$ is uniquely determined by $\mid_\B(L)$ and  $C' \cdot L$.  Since $C$ is parallel to $A$, it follows that $C$ is the unique line for which, given $C'$,  $\B \cup \{\{C,C'\}\}$ is a bisector arrangement.

\smallskip

{\textsc{Claim 4.}} Statement ($*$)  is valid if at most one of the bisectors in $\B$    has an undetermined midpoint. 

\smallskip

By Claim 2, if  three bisectors in $\B$ have finite midpoints, then statement ($*$) is valid. If fewer than three   have finite midpoints, then by  Claim 1 and the assumption that there is at most one undetermined midpoint in $\B$, at least one of the lines in $\B$ has a finite midpoint and at least one has an infinite midpoint, in which case statement ($*$) is valid  by Claim 3.

\smallskip

{\textsc{Claim 5.}} Statement ($*$) is valid if exactly two  bisectors in $\B$ have  undetermined midpoints.

\smallskip

First we claim that at most one of $A,A'$ and at most one of $B,B'$ have an undetermined midpoint. For if $\mid_\B(A)$ and $\mid_\B(A')$  are undetermined, then (a) $A \in \{B,B'\}$ or $A$, $B$ and $B'$ are parallel, and (b) $A' \in \{B,B'\}$ or $A'$, $B$ and $B'$ are parallel. If $A, B,B'$ are all parallel, then (b) implies $A'$ is also parallel to these lines, contrary to the assumption that $\B$ is a nontrivial bisector arrangement. Thus $A \in \{B,B'\}$, and similarly, $A' \in \{B,B'\}$, which again contradicts 
 the fact that $\B$ is nontrivial.  Thus  $\mid_\B(A)$ and $\mid_\B(A')$ cannot both be  undetermined. Similarly, $\mid_\B(B)$ and $\mid_\B(B')$ cannot both be  undetermined.

Without loss of generality, $\mid_\B(A)$ and $\mid_\B(B)$ are undetermined. By Claim 3, we may reduce to the case that  $A'$ and $B'$ both have infinite midpoints or both have finite midpoints. We will rule out the former case. 
Suppose by way of contradiction that 
$\mid_\B(A')$ and $ \mid_\B(B')$ are infinite. 
Then $A'$ is parallel to exactly one of $B,B'$, and $B'$ is parallel to exactly one of $A,A'$. This implies $A'$ is parallel to $B'$; $A'$ is not parallel to $B$;  and $B'$ is not parallel to $A$.  However, $\mid_\B(A)$ is undetermined, so $A$ must be parallel to $B$ and $B'$ or $A \in \{B,B'\}$. This forces $A =B$, which since $A'$ is parallel to $B'$ contradicts the assumption that $\B$ is a nontrivial bisector arrangement. 
This shows that $\mid_\B(A') $ and $ \mid_\B(B')$ cannot both be infinite. 
Therefore, $\mid_\B(A')$ and $\mid_\B(B')$ are finite. 

Next, since $\mid_\B(A)$ is undetermined, either $A$ is parallel to $B$ and $B'$ or $A \in \{B,B'\}$. Since $\mid_\B(B')$ is finite, $A$ is not parallel to $B'$, and so $A  = B$.  
We claim at least one of the two lines $C,C'$ is not the line $A=B$. Suppose otherwise. The fact that $A'$ and $B'$ bisect $C,C'$ imply $\mid_\B(A')$ and $\mid_\B(B')$ lie on the line $A = B = C = C'$, which since
$A'$ bisects $B,B'$ with midpoint $\mid_\B(A')$ implies $\mid_\B(A') = \mid_\B(B')$, with this point lying on $A = B = C = C'$.  But then
  all four lines $A,A',B,B'$ go through the same point, 
 contrary to the assumption that $\B$ is a nontrivial bisector arrangement.  Thus we may assume without loss of generality that $C$ is not the line $A =B$.  

We will show $C'$ is uniquely determined by $\B$ and $C$. 
Let $\Lambda_{A'}$ denote reflection on the line $A'$ about $\mid_\B(A')$, and let $\Lambda_{B'}$ denote reflection on   $B'$ about $\mid_\B(B')$. Since $\B$ is a nontrivial bisector configuration and $A=B$, the lines $A'$ and $B'$ cannot be parallel.  
Also,  $\Lambda_{A'}$ and $\Lambda_{B'}$ are uniquely determined by $\B$.  
Since $\Lambda_{A'}(A' \cdot C) = A' \cdot C'$ and $\Lambda_{B'}(B' \cdot C) = B' \cdot C'$, to prove 
 that $C'$ is determined by $\B$ and $C$, it suffices to show that $A' \cdot C' \ne B' \cdot C'$ since these two points (which are entirely determined by $\B$ and $C$ via the involutions $\Lambda_{A'}$ and $\Lambda_{B'}$) then determine $C'$. 
If  
 $A' \cdot C' = B' \cdot C'$, 
then  $A' \cdot C' = B' \cdot C' = A' \cdot B'$, so that 
\begin{eqnarray*}
C\cdot A' & = &  \Lambda_{A'}( A' \cdot C') \:\: = \:\: \Lambda_{A'}(A' \cdot B') \:\:= \:\:A' \cdot B \:\: \in \:\: B \\
C \cdot B' & = &  \Lambda_{B'}(B' \cdot C') \:\:= \:\: \Lambda_{B'}(A' \cdot B') \:\:= \:\:B' \cdot A \:\:\in \:\:A.
\end{eqnarray*} 
  But since $C$ is not the line $A=B$, this implies $A'$ and $B'$ meet at the same point on $A = B$, contradicting the fact that 
   $\B$ is a nontrivial bisector configuration.  
   %
 Therefore,  $A' \cdot C' \ne B' \cdot C'$ and the proof of Claim 5 is complete.



\smallskip

{\textsc{Claim 6.}} Statement ($*$) is true in all cases. 

\smallskip

If at most one of the bisectors in $\B$ has an undetermined midpoint, then statement ($*$) is true by Claim 4. If exactly two of these lines has an undetermined midpoint, then statement ($*$) is true by Claim 5. Thus it suffices to observe that  three of the four lines in $\B$ cannot have undetermined midpoints. This is  so because in such a case both lines in one of the two pairs $A,A'$ and $B,B'$ must have undetermined midpoints, and, as argued at the beginning of the proof of Claim 5, this leads to a contradiction to the assumption that $\B$ is a nontrivial bisector arrangement. Thus 
statement ($*$) is true in all cases, which proves the lemma. 
\end{proof}


With Lemma~\ref{unique}, we  can now prove the main theorem of the article. 

\begin{theorem} \label{main theorem 1} A set  of line pairs is a bisector field if and only if it is a nontrivial asymptotic pencil. 
%

\end{theorem}

\begin{proof}
 Suppose   that $\B$ is   a nontrivial asymptotic pencil.  
  By Proposition~\ref{quad exists}, there is a quadrilateral $Q$ in $\B$ such that $\B $ is the asymptotic pencil of  $Q$.   
  If a line $\ell$ is a line in a pair in $\B$,  then $\ell$ bisects $Q$ by Lemma~\ref{main lemma}, and 
 so $\ell$ bisects $\B$ by 
  Corollary~\ref{amazing cor}. 
  Therefore,   $\B$  is a bisector arrangement. 
  To see next that  $\B$ is a bisector field, let $\PP= \{\ell_1,\ell_2\}$ be a pair of lines such that $\B \cup \{\PP\}$ is a bisector arrangement.  
 We show  $\PP \in \B$. 

There are two independent reducible quadratics $f_1$ and $f_2$ whose zero sets are the pairs of opposite sides of $Q$.  By Lemma~\ref{generators}, $f_1$ and $f_2$ generate the asymptotic pencil $\B$.  Since $f_1$ and $f_2$ are reducible conics, every line meets the pairs of lines defined by $f_1$ and $f_2$ (possibly at infinity), and so since 
by assumption  $\ell_1$ and $\ell_2$ bisect $\B$ and hence bisect $\{f_1,f_2\}$, 
Lemma~\ref{main lemma} implies there are 
 line pairs $\PP_1$ and $\PP_2$ in $\B$ such that $\ell_1 \in \PP_1$ and $\ell_2 \in \PP_2$.  
Since $\B \cup \{\PP\}$ is a bisector arrangement extending $\B$, the collection $\{\PP_1,\PP_2,\PP\}$ is a bisector arrangement extending the bisector arrangement $\{\PP_1,\PP_2\}$. 
 Since $\PP_1$ and $\PP$ share $\ell_1$ and $\PP_2$ and $\PP$ share $\ell_2$, Lemma~\ref{unique} implies $\PP=\PP_1$ or $\PP = \PP_2$. Thus $\PP \in \B$, as claimed.

Conversely, suppose $\B$ is a bisector field. We claim first that $\B$ contains two pairs of lines that are defined by 
 independent quadratics.  If $\B$ contains a degenerate parabola defined by a quadratic $f_1$, then since $\B$ is a nontrivial bisector arrangement, there is a reducible quadratic $f_2$ with zero set in $\B$ such that not all four linear components of $f_1$ and $f_2$ are parallel. Thus $f_1$ and $f_2$ 
 are independent. 
If $\B$ does not contain a degenerate parabola, then choose $f_1$ to be any degenerate hyperbola with zero set in $\B$ and, using the fact that $\B$ is a nontrivial bisector arrangement, choose $f_2$  to be a degenerate hyperbola in $\B$ such that $V(f_1) \ne V(f_2)$. If $f_1$ is not a translation of $f_2$, then $f_1$ and $f_2$ are independent quadratics defining line pairs in $\B$. 
If, however, $f_2$ is a translation of $f_1$, then since $\B$ is nontrivial and by assumption consists only of hyperbolas, there is a hyperbola $f_3$ that is not a translation of $f_1$ or $f_2$. Choose whichever of $f_1$ and $f_2$ does not share a center with $f_3$, and this quadratic and $f_3$ will be independent quadratics that define line pairs in $B$.    

In all cases, we have found a pair of independent quadratics $f_1$ and $f_2$ that define line pairs in $\B$ and  have the additional property that if $f_1$ and $f_2$ are hyperbolas, then $f_1$ and $f_2$ do not share the same center. This  implies
the line pairs defined by $f_1$ and $f_2$ 
  form a nontrivial bisector arrangement contained in $\B$. 
  
    We claim $\B $ is the asymptotic pencil generated by $f_1$ and $f_2$. 
Let $f $ be a quadratic that defines a line pair in $ \B$. The line pairs that are the zero sets of $f_1,$ $f_2$ and $f$
 form
 a bisector arrangement, and so by Lemma~\ref{main lemma}, the linear components of $f$ are linear components of reducible quadratics $g_1$ and $g_2$ in 
$\k f_1 + \k f_2+ \k$. We have established in the first part of the proof that the set of  reduced conics in $\k f_1 + \k f_2+ \k$, i.e., the asymptotic pencil generated by $f_1$ and $f_2$, is the set of line pairs of a bisector field. Thus the line pairs defined by $f_1,f_2,g_1$ 
 form
 a bisector arrangement, as do those of $f_1,f_2,g_2$.    Since the bisector arrangement consisting of the line pairs of $f_1$ and $f_2$ 
  is nontrivial and $f$ and $g_1$ share a linear component, and $f$ and $g_2$ share the other linear component of $f$,  Lemma~\ref{unique} implies $V(f)= V(g_1)$ or $V(f) = V(g_2)$. Either way,    $f \in  \k f_1 + \k f_2+\k$, and hence 
  the pair of lines defined by $f$ is in 
  the asymptotic pencil generated by $f_1,f_2$, which shows the pairs of lines in  $\B$ are in this asymptotic pencil.  
 Therefore,
  since $\B$ is a bisector field contained in a bisector arrangement consisting of the set of line pairs in the asymptotic pencil generated by $f_1,f_2$, the maximality of $\B$ implies these two arrangements are equal, 
  which proves the theorem. 
\end{proof}

  \medskip

{\noindent}{\it Funding.}
No financial support was received for this manuscript.

\smallskip



{\noindent}{\it Conflict of interest.} The authors declare that there is no conflict of interest in this work.

\smallskip

{\noindent}{\it Data availability.} Not applicable.


\begin{thebibliography}{99}
  
  
  
  
  
  
  
  






\bibitem{Beltrami} E.~Beltrami,  Intorno ad alcuni teoremi di Feuerbach e di Steiner. Esercitazione analitica. 
Memorie R.~Accad.~Sci.~Ist.~Cl.~Sci.~fis.~Bologna, 
serie III, tomo V, (1874), 543--566.

\bibitem{BC1} A.~Berele, and S.~Catoiu, Bisecting envelopes of convex polygons, Adv. Appl. Math. 137 (2022), 102342.

\bibitem{BC2} A.~Berele, and S.~Catoiu, Extreme area and perimeter bisectors in a triangle, Aequat. Math. (2023), \url{https://doi.org/10.1007/s00010-023-01001-9}

\bibitem{Berger} M.~Berger, {\it Geometry II}, Springer-Verlag, 1987.













\bibitem{OW5} B.~Olberding and E.~Walker, 
Bisector fields of quadrilaterals, submitted, \url{
https://doi.org/10.48550/arXiv.2305.11762}. 

\bibitem{OW6} B.~Olberding and E.~Walker, 
Bisector fields and projective duality, submitted, \url{
https://doi.org/10.48550/arXiv.2306.08612}. 

\bibitem{OWSpider} B.~Olberding and E.~Walker, Diametric sets of collections of pairs of lines, preprint. 

\bibitem{Vac} M. Vaccaro, Historical origins of the nine-point conic—the contribution of Eugenio Beltrami, Hist.
Math. 51 (2020), 26–48.







 
\end{thebibliography}
\end{document}